\documentclass[11pt,a4paper]{amsart}
\usepackage{graphicx,multirow,array,amsmath,amssymb}
\usepackage{xcolor}
\usepackage{longtable}
\usepackage{supertabular}

\newtheorem{theorem}{Theorem}

\newtheorem{lemma}[theorem]{Lemma}
\newtheorem{corollary}[theorem]{Corollary}
\theoremstyle{definition}
\newtheorem*{definition}{Definition}
\theoremstyle{remark}

\begin{document}
	
\title[New Upper Bounds on the Ribbonlength of Alternating Links]{New Upper Bounds on the Ribbonlength of Alternating Links with Bipartite Dual Graphs}

\author[H. Yoo]{Hyungkee Yoo}
\address{Department of Mathematics Education, Sunchon National University, Suncheon 57922, Republic of Korea}
\email{hyungkee@scnu.ac.kr}


\begin{abstract}
The ribbonlength of a link is a geometric invariant defined as the infimum of the ratio of the length to the width of a folded ribbon realization of the link.
In this paper, we prove that if an alternating link admits an alternating diagram with a bipartite dual graph, then its ribbonlength satisfies
$$
\mathrm{Rib}(L) \le \sqrt{3} \, c(L).
$$
Using this result, we present improved upper bounds on the ribbonlength for several knots and links with small crossing numbers, and determines the exact ribbonlength of the Hopf link to be $2\sqrt{3}$.
\end{abstract}
	
\maketitle

\section{Introduction} \label{sec:intro}

A \textit{link} is a one-dimensional closed manifold embedded in three-dimensional space $\mathbb{R}^3$.
A link with a single connected component is called a \textit{knot}.
To measure the geometric complexity of knots and links, the minimal length required to realize a link as a rope of unit thickness, called the \textit{ropelength}, has been extensively studied.
In 2005, Kauffman\cite{K} introduced the \textit{ribbonlength}, the two-dimensional version of the ropelength.
A \textit{folded ribbon link}  is obtained by folding a narrow rectangular strip into a planar shape that traces a polygonal link with finitely many straight segments, as illustrated in Figure~\ref{fig:folded}.
In 2015, Denne et al~\cite{DKTZ} gave a rigorous definition of folded ribbon knots as follows.

\begin{definition}
Let $K$ be the knot. 
Then $K_w$ is a folded ribbon knot of $K$  with width $w$ that satisfies following two conditions:
    \begin{enumerate}
        \item The ribbon is flat and its fold lines are disjoint.
        \item The core of $K_w$ is a union of a finite number of consecutively straight lines with crossing information which represents the knot type of $K$.
    \end{enumerate}
\end{definition}

\begin{figure}[h!]
	\includegraphics{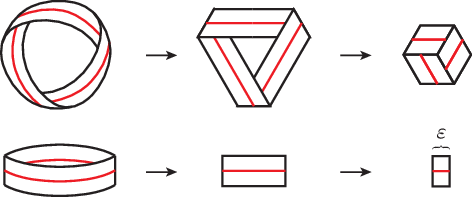}
	\caption{Two folded ribbon knots of the trivial knot}
	\label{fig:folded}
\end{figure}

Note that, for the trivial knot, a folded knot can be constructed by placing two congruent rectangles of width 1 and length $\varepsilon$, for an arbitrary $\varepsilon>0$, on top of each other.
Gluing together their sides of length 1 yields the desired folded knot, as drawn in Figure~\ref{fig:folded}.
In this construction, the ribbonlength can be made arbitrarily small.
For this reason, we define the folded ribbonlength as an infimum.

\begin{definition}
Let $K$ be a knot (or link) and $w$ be a positive real number.
Then
$$Rib(K)=\inf_{K' \in [K]_w}\frac{\emph{\text{Len}}(K')}{w}$$
is called a folded ribbonlength of $K$.
Here, $[K]_w$ denotes the set of folded ribbon knots of width $w$
representing the knot type $K$, $\emph{\text{Len}}(K')$ is the length of the core of $K'$,
and we do not distinguish whether the ribbon is one-sided or two-sided.
\end{definition}

A \textit{link diagram} is a regular projection of a link onto the plane $\mathbb{R}^2$, where each crossing is equipped with over/under-crossing data.
The \textit{crossing number} $c(L)$ is defined as the minimal number of crossings among all link diagrams representing $L$.
Kusner conjectured in 2005 that the ribbonlength admits a linear upper bound in terms of the crossing number.
For many classes of knots and links, Kusner's conjecture has been verified~\cite{DHLM, KMRT, KNY1, KNY2}.
Tian~\cite{T} later established a quadratic upper bound, which was further improved by Denne~\cite{D} in 2020.
Recently, Kim, No, and Yoo~\cite{KNY3} proved that
$$\mathrm{Rib}(L)\le \frac{5}{2} \, c(L)+1,$$
thereby confirming Kusner’s conjecture~\cite{D}.
Meanwhile, Diao and Kusner conjectured that the lower bound of the ribbonlength in terms of the crossing number is sublinear.
More recently, Denne and Patterson~\cite{DP} proved that this lower bound is, in fact, a constant.

The purpose of this paper is to further improve this upper bound for certain special classes of links satisfying additional combinatorial conditions.
Our approach combines geometric constructions using equilateral triangles with diagrammatic techniques involving checkerboard dual graphs associated to link diagrams.
In particular, we focus on alternating links that admit minimal crossing diagrams whose dual graphs are bipartite.

Our main result is the following.

\begin{theorem} \label{thm:main}
If an alternating link $L$ has an alternating diagram with a bipartite dual graph, then
$$\mathrm{Rib}(L) \leq \sqrt{3} c(L).$$
\end{theorem}

This result strengthens the general linear bound of Kim, No, and Yoo for this class of links.
It yields improved upper bounds for knots and links with small crossing numbers whose ribbonlengths were previously studied, and determines the exact ribbonlength of the Hopf link to be $2\sqrt{3}$.
In addition, we show that this upper bound is preserved under connected sum operations, providing the first explicit estimate for the ribbonlength of connected sums under the above assumptions.

This paper is organized as follows.
In Section~\ref{sec:example}, we observe folded ribbon realizations of the Hopf link, which serve as a motivating example.
Section~\ref{sec:circular} introduces circular three-page presentations, a key tool in our constructions.
In Section~\ref{sec:alt}, we deal with alternating diagrams with bipartite dual graphs.
Finally, in Section~\ref{sec:proof}, we prove of Theorem~\ref{thm:main}.
Using this result, we present improved upper bounds on the ribbonlength for several knots and links and compare them with previously known estimates.

\section{Example: the Hopf link} \label{sec:example}

In this section, we investigate how the Hopf link can be realized as a folded ribbon link.
The Hopf link is the simplest nontrivial link consisting of two components that are linked exactly once.
In 2020, Denne~\cite{D} proposed a folded ribbon construction for the Hopf link, as illustrated in Figure~\ref{fig:Hopf1}.
In this realization, the length of each component is exactly twice its width.

\begin{figure}[h!]
	\includegraphics{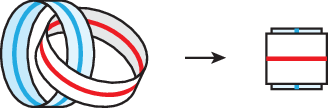}
	\caption{A folded ribbon link of the Hopf link}
	\label{fig:Hopf1}
\end{figure}

We now present a shorter folded ribbon construction of the Hopf link.
Consider two rectangles of unit width and length $\sqrt{3}$.
These rectangles are placed in a staggered configuration so that their diagonals intersect, as shown in Figure~\ref{fig:Hopf2}.
By sequentially folding the excess parts inward, we obtain a folded ribbon link whose shape consists of stacked equilateral triangles.
Note that this construction consists of six equilateral triangles, each contributing a ribbonlength of $\frac{1}{\sqrt{3}}$.
Hence, the total ribbonlength is $2\sqrt{3}$, which is shorter than the value $4$ obtained in Denne’s original construction.

\begin{figure}[h!]
	\includegraphics{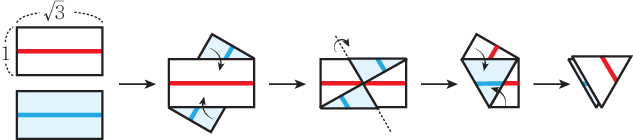}
	\caption{A new folded ribbon Hopf link}
	\label{fig:Hopf2}
\end{figure}

An exaggerated depiction of the folding process is shown in Figure~\ref{fig:Hopf3}.

\begin{figure}[h!]
	\includegraphics{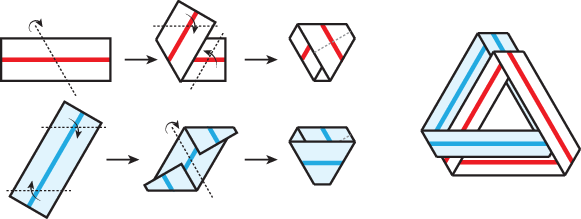}
	\caption{Construction process for the folded ribbon Hopf link}
	\label{fig:Hopf3}
\end{figure}

It is noteworthy that, under this construction, each component of the folded ribbon link realizes a M\"{o}bius band structure.
Recently, Schwartz~\cite{Sch} proved that a smooth embedded paper M\"{o}bius band must have an aspect ratio greater than $\sqrt{3}$.
Using this fact, we obtain the following theorem.

\begin{theorem}
The ribbonlength of the Hopf link is $2\sqrt{3}$.
\end{theorem}

\section{Rotated circular three-page presentations} \label{sec:circular}

In 1995, Cromwell~\cite{Cr} showed that any link can be embedded in a finite collection of half-planes sharing a common axis, known as the \textit{binding axis}, within an open-book decomposition of $\mathbb{R}^3$.
These half-planes are referred to as \textit{pages}.
When each page contains exactly one arc of the link, the embedding is called an \textit{arc presentation}.
Such presentations are particularly useful for investigating link invariants, as they offer a combinatorial method for analyzing the structure of knots and links.
If we modify this setup to allow multiple disjoint arcs per page while restricting the total number of pages to three, the result is called a \textit{three-page presentation}.

\begin{figure}[h!]
	\includegraphics{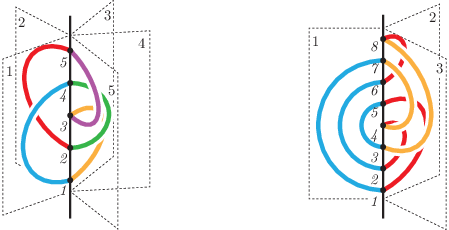}
	\caption{An arc presentation and a three-page presentation of the trefoil knot}
	\label{fig:arc}
\end{figure}

In 1999, Dynnikov~\cite{Dyn} showed that every link can be represented by a three-page presentation.
This presentation has been widely studied not only for knots and links, but also for more general objects such as 3-valent spatial graphs and singular knots~\cite{Dyn, Dyn2, Dyn3, JLY, Kur, Kur2, KV, Y}.
In both arc and three-page presentations, each arc is combinatorially described by a binding index and a page number.
The binding axis intersects the link at finitely many points, called binding points, which are typically indexed sequentially from bottom to top as $1, 2, \dots, n$ and referred to as \textit{binding indices}.
Each arc is also labeled with a \textit{page number}, indicating the page to which it belongs in the open-book decomposition.

Consider the one-point compactification of $\mathbb{R}^3$ equipped with an open-book structure.
Then the binding axis transforms into an unknotted circle, and each half-plane becomes a disk.
This circle is called the \textit{binding circle}.
In this context, the notion of a \textit{circular three-page presentation} was introduced in~\cite{Y}.
The binding index, originally represented by vertical height, is now assigned sequentially along the binding circle.

A circular three-page presentation can be represented as a link diagram, which provides a convenient planar description of the presentation.
First, we choose a simple closed curve as a binding circle.
If the first and second pages are placed inside the binding circle and the third page is placed outside, then all crossings can be arranged to lie within the interior of the binding circle.
In this configuration, the under-strands are positioned on the first page, while the over-strands lie on the second page, as illustrated in Figure~\ref{fig:three}.

\begin{figure}[h!]
	\includegraphics{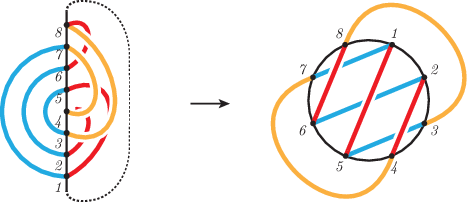}
	\caption{A three-page presentation and its circular three-page presentation}
	\label{fig:three}
\end{figure}

To formalize this planar description, we introduce the following definition.

\begin{definition}
Let $D$ be a link diagram.
A simple closed curve $\gamma$ is called a \textit{binding circle for a three-page presentation} of $D$ if it satisfies the following conditions:
\begin{enumerate}
    \item The curve $\gamma$ intersects $D$ in finitely many points.
    \item All crossings of $D$ lie inside $\gamma$.
    \item Each arc of $D$ cut by $\gamma$ is of one of the following three types:
    \begin{enumerate}
        \item it lies outside $\gamma$;
        \item it lies inside $\gamma$ and participates only in over-crossings;
        \item it lies inside $\gamma$ and participates only in under-crossings.
    \end{enumerate}
    \item No two arcs of the same type are adjacent along $\gamma$.
\end{enumerate}
\end{definition}

Given an arbitrary link diagram, the existence of such a binding circle can be verified by a straightforward argument.
First, we ignore the crossing information of the diagram and regard it as a $4$-valent planar graph.
If this graph is connected, we choose a spanning tree and consider the boundary of a sufficiently small $\varepsilon$-neighborhood of the tree in $\mathbb{R}^2$.
It is then easy to check that this boundary curve satisfies the conditions in the definition of a binding circle for a three-page presentation.
If the underlying graph is not connected, one may construct a binding circle for each connected component and then join them together by taking their connected sum, which again yields a curve satisfying the required properties.
Further details will be presented in a subsequent work.

\begin{lemma}\label{lem:binding}
Let $D$ be a link diagram.
If there exists a binding circle for a three-page presentation of $D$ with $n$ binding points, then $D$ admits a three-page presentation consisting of $n$ arcs.
\end{lemma}

Now we recall the stacked equilateral triangle configuration of the Hopf link introduced in Section~\ref{sec:example}.
Then a three-page presentation can be constructed by aligning the triangles along an axis that passes through their centroids and connecting them with arcs according to their adjacency relations as drawn in Figure~\ref{fig:trans}.
We observe that each equilateral triangle corresponds to a binding point.
Assuming the ribbon has unit width, the length of ribbon represented by each equilateral triangle is $\frac{1}{\sqrt{3}}$.
Hence, the total length of the resulting folded ribbon knot is given by $\frac{1}{\sqrt{3}}$ times the number of binding points.

\begin{figure}[h!]
	\includegraphics{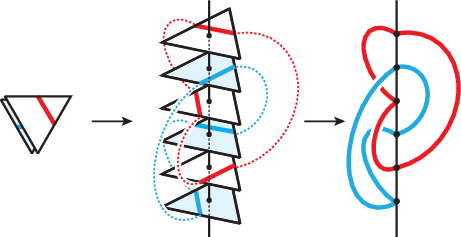}
	\caption{Conversion from a folded ribbon link to a three-page presentation}
	\label{fig:trans}
\end{figure}

As in the example of the Hopf link, let us consider the process of constructing a link by folding a ribbon along equilateral triangle shapes.
When the ribbon is wrapped along each component, it becomes apparent that each edge is folded sequentially along the sides of the triangles.
This observation implies that, in the corresponding three-page presentation, the arcs of each component traverse the pages in a cyclic order as drawn in Figure~\ref{fig:rule}.

\begin{figure}[h!]
	\includegraphics{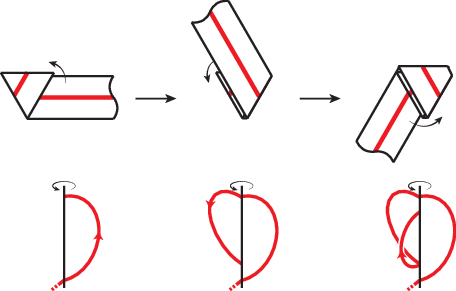}
	\caption{Details of the conversion process}
	\label{fig:rule}
\end{figure}

Based on this cyclic behavior, we introduce the following definition.

\begin{definition}
A three-page presentation is \textit{rotated} if the arcs of each component traverse the pages cyclically.
\end{definition}

Note that the cyclic order in which the arcs of each component traverse the pages is not fixed.
In particular, reversing this order does not affect the definition of a rotated three-page presentation.
It is immediate that the number of arcs in a rotated three-page presentation is a multiple of three.
We now introduce a key lemma that will play a central role in the proof of our main results.

\begin{lemma} \label{lem:length}
Let $L$ be a link that admits a rotated three-page presentation consisting of $3n$ arcs. Then $L$ can be realized as a folded ribbon link of length $\sqrt{3}n$ with unit width.
\end{lemma}

\begin{proof}
Suppose that we are given a rotated three-page presentation.
For each binding point, we place an equilateral triangle centered on the binding axis.
The connectivity among these triangles is then determined by the arcs in the rotated three-page presentation.
As observed in the previously discussed Hopf link example, each equilateral triangle corresponds to a ribbon segment of length $\frac{1}{\sqrt{3}}$ when the ribbon has unit width.
Therefore, this construction yields a folded ribbon realization whose total ribbonlength is equal to $\frac{1}{\sqrt{3}}$ times the number of binding points.
\end{proof}

As the triangles are connected sequentially, the front and back sides of the ribbon are determined alternately.
Consequently, if the number of triangles forming a component of the link is odd, the corresponding ribbon is one-sided, whereas if it is even, the ribbon is two-sided.

\section{Alternating diagrams with bipartite dual graphs} \label{sec:alt}

In this section, we introduce definitions and basic properties related to alternating diagrams with bipartite dual graphs.
A link is said to be \textit{alternating} if it admits a diagram in which the crossings alternate between over and under along each component of the link.
In a link diagram, a crossing is called \textit{nugatory} if there exists a simple closed curve that passes through the crossing and does not intersect the diagram elsewhere.
A link diagram is called \textit{reduced} if it has no nugatory crossings.
Using techniques based on the Jones polynomial,
Kauffman~\cite{K2}, Murasugi~\cite{M}, and Thistlethwaite~\cite{Th} independently proved in 1987
that any reduced alternating diagram realizes the minimal crossing number of the corresponding link.

Given a link diagram $D$, we apply a checkerboard coloring to the regions of the diagram.
The \textit{dual graph} $\Gamma(D)$ is constructed by placing a vertex in each shaded region of $D$ and connecting two vertices by an edge whenever the corresponding regions meet at a crossing.
Note that $\Gamma(D)$ is, in general, a multigraph and may contain loops or multiple edges.
Each edge of $\Gamma(D)$ is assigned a sign determined by the local crossing information of $D$.
For a detailed description of this construction and the associated sign conventions, we refer the reader to Adams's book~\cite{A}.

If a given link diagram $D$ is alternating, then every edge of the dual graph $\Gamma(D)$ has the same sign.
Without loss of generality, we may assume that all edges carry the positive sign, since otherwise we can replace $D$ with its mirror image.
This assumption is justified because the ribbonlength is invariant under taking mirror images.
Therefore, the signs on the edges of the dual graph can be ignored when the diagram is alternating.

A bipartite multigraph is a graph whose vertex set is partitioned into two disjoint subsets, such that edges connect only vertices from different subsets.
Multiple edges between the same pair of vertices are allowed, but no edges may connect two vertices within the same subset.
Thus there are no loops in a bipartite multigraph.
In particular, it contains no odd cycles.
Our goal is to establish an improved upper bound on the ribbonlength for alternating links that admit a minimal crossing diagram whose dual graph is bipartite as drawn in Figure~\ref{fig:bipartite}.

\begin{figure}[h!]
	\includegraphics{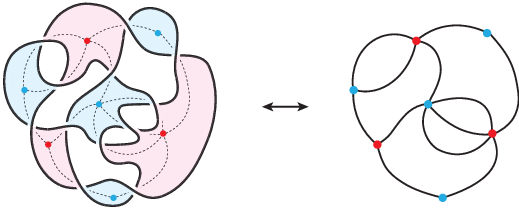}
	\caption{A minimal crossing alternating diagram with a bipartite dual graph}
	\label{fig:bipartite}
\end{figure}

Now we assume that there is a nugatory crossing in $D$.
The edge of $\Gamma(D)$ corresponding to this crossing is either a loop or a cut-edge, as drawn in Figure~\ref{fig:dual}.
If the crossing gives the unique connection between two parts of $\Gamma(D)$, then the corresponding edge is a cut-edge, that is, an edge whose removal disconnects the graph.
Otherwise, if the two incident shaded regions coincide, then the edge is a loop.
Conversely, if $\Gamma(D)$ contains a loop or a cut-edge, then the associated crossing admits a simple closed curve intersecting $D$ only at that crossing and is therefore nugatory.
In particular, the diagram $D$ is reduced if and only if $\Gamma(D)$ has no loops and no cut-edges.

\begin{figure}[h!]
	\includegraphics{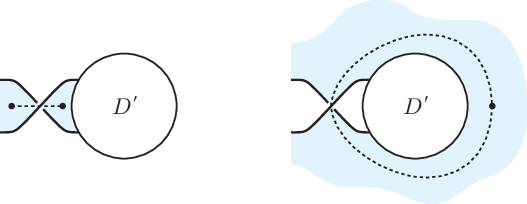}
	\caption{Converting rule from link diagram to its dual graph}
	\label{fig:dual}
\end{figure}

The \textit{connected sum} $K_1 \# K_2$ of two knots $K_1$ and $K_2$ is a fundamental operation in knot theory that produces a new knot by joining the two knots along short arcs.
More precisely, one removes a small open arc from each knot and connects the resulting endpoints by two arcs in an orientation-preserving manner.
The resulting knot type is well defined up to ambient isotopy and is independent of the specific choices made in the construction.
This operation extends naturally to links by forming connected sums of selected components.
In the case of alternating knots, the crossing number is additive under the connected sum operation.

We now examine how the property of admitting an alternating diagram with a bipartite dual graph behaves under the connected sum operation.

\begin{lemma} \label{lem:conn}
Let $L_1$ and $L_2$ be alternating links, each admitting an alternating diagram whose dual graph is bipartite.
If the connected sum $L_1 \# L_2$ is alternating, then it also admits an alternating diagram whose dual graph is bipartite.
\end{lemma}

\begin{proof}
Let $D_1$ and $D_2$ be alternating diagrams of $L_1$ and $L_2$, respectively, whose dual graphs $\Gamma(D_1)$ and $\Gamma(D_2)$ are bipartite.
For each $i=1,2$, choose a small disk in the plane that meets $D_i$ in a single embedded arc and contains no crossings.
A connected sum diagram $D$ of $L_1 \# L_2$ is obtained by removing these arcs and gluing the resulting endpoints in the standard way.
By assumption, $L_1 \# L_2$ is alternating, and we choose the gluing so that $D$ is an alternating diagram of $L_1 \# L_2$.

Under this operation, the dual graph $\Gamma(D)$ is obtained by identifying a vertex $v_1$ of $\Gamma(D_1)$ with a vertex $v_2$ of $\Gamma(D_2)$.
If $v_1$ and $v_2$ have the same color in the bipartitions of $\Gamma(D_1)$ and $\Gamma(D_2)$, then $\Gamma(D)$ is immediately bipartite.
If $v_1$ and $v_2$ have different colors, we interchange the two parts of the bipartition of $\Gamma(D_2)$.
Since this operation preserves bipartiteness, the resulting graph $\Gamma(D)$ is bipartite in either case.
\end{proof}

\begin{figure}[h!]
	\includegraphics{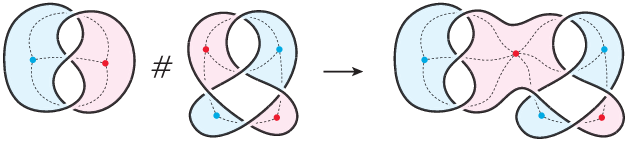}
	\caption{connected sum}
	\label{fig:conn}
\end{figure}

Theorem~\ref{thm:main} concerns an upper bound on the ribbonlength in terms of the crossing number.
Therefore, it is essential to work with minimal crossing diagrams.
The following lemma guarantees that, given any alternating diagram whose dual graph is bipartite, there exists a minimal crossing diagram with the same property.

\begin{lemma} \label{lem:minimal}
If a link has an alternating diagram whose dual graph is bipartite, then it admits a minimal crossing diagram whose dual graph is also bipartite.
\end{lemma}

\begin{proof}
Let $D$ be an alternating diagram whose dual $\Gamma(D)$ is bipartite.
If $D$ is reduced, then $D$ is a minimal crossing diagram by the results of Kauffman, Murasugi, and Thistlethwaite, and we are done.

Otherwise, suppose that $D$ has a nugatory crossing $c$.
Let $e_c$ be the edge of $\Gamma(D)$ corresponding to $c$.
Since $\Gamma(D)$ is bipartite, $e_c$ is not a loop.
Because $c$ is nugatory, $e_c$ is a cut-edge in $\Gamma(D)$.
Removing $c$ from $D$ corresponds to contracting $e_c$ in $\Gamma(D)$.
Since $e_c$ is a cut-edge, it is not contained in any cycle of $\Gamma(D)$.
Therefore contracting $e_c$ cannot create a new odd cycle, and the resulting dual graph remains bipartite.
The diagram obtained by removing $c$ from $D$ is still alternating and has fewer crossings.
Iterating this process terminates at a reduced alternating diagram whose dual graph is bipartite.
\end{proof}

By the lemma, it suffices to consider alternating diagrams with bipartite dual graphs, without assuming minimality.

\section{Construction of Folded Ribbon Links} \label{sec:proof}
In this section, we investigate binding circles for rotated circular three-page presentations associated with alternating diagrams whose dual graphs are bipartite.
To complete the proof of Theorem~\ref{thm:main}, it suffices to prove the following theorem.

\begin{theorem} \label{thm:main2}
Let $L$ be an alternating link, and let $D$ be an alternating link diagram of $L$ with $n$ crossings.
If the dual graph $\Gamma(D)$ is bipartite, then $L$ admits a rotated circular three-page presentation consisting of exactly $3n$ arcs.
\end{theorem}

\begin{proof}
Let $D$ be an alternating diagram of a link $L$ with $n$ crossings whose dual graph is bipartite, as assumed in the theorem.
We color the vertices of $\Gamma(D)$ red and blue according to its bipartition.
Our goal is to construct a suitable binding circle in $D$ that yields a rotated circular three-page presentation of $L$.
Then Lemma~\ref{lem:binding} guarantees that the resulting diagram is a rotated circular three-page presentation.

First, we consider the case that $\Gamma(D)$ is connected.
Assign to each crossing a sufficiently small disk centered at the crossing so that these disks are pairwise disjoint.
Without loss of generality, choose the red color class in the bipartition of $\Gamma(D)$ and merge the adjacent disks along the corresponding regions.
This results in a collection of disjoint disks corresponding to the red vertices of $\Gamma(D)$.
Each disk is then perturbed slightly so that it touches each side of the original region exactly once as drawn in Figure~\ref{fig:area}.

\begin{figure}[h!]
	\includegraphics{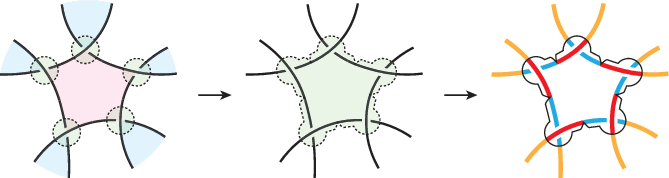}
	\caption{Constructing a circle around the chosen region}
	\label{fig:area}
\end{figure}

With this construction, every crossing of the diagram lies in the interior of one of these disks.
When the link diagram is cut along the boundaries of these disks, it is decomposed into over–strands, under–strands, and arcs lying outside the collection of disks.
For each crossing, there are exactly two arcs contained in the interior of the disks, one corresponding to an over–strand and the other to an under–strand.
Consequently, there are $2n$ such arcs lying inside the disks.
Together with the $n$ arcs lying outside the collection of disks, this decomposition produces a total of $3n$ arcs.
Moreover, since $D$ is alternating, when one traverses the link diagram following a suitable orientation, the arcs appear in a repeating pattern of lying outside the disks, passing as over–strands, and passing as under–strands.
Thus, the remaining task is to merge these disks appropriately so as to obtain a binding circle of a rotated circular three-page presentation.

The dual graph $\Gamma(D)$ is connected, and hence admits a spanning tree.
We fix such a spanning tree and merge the disks in the collection by performing connected sum operations successively along its edges.
As a result, we obtain a single disk that retains all the properties established above.
Therefore, the boundary of this disk is the desired binding circle as shown in Figure~\ref{fig:entire}.

\begin{figure}[h!]
	\includegraphics{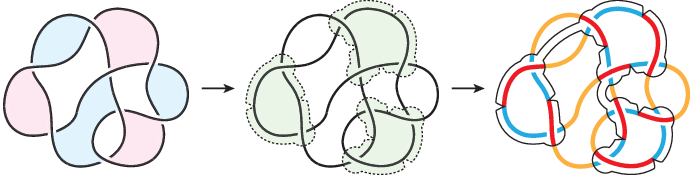}
	\caption{Connecting the circles to obtain a binding circle}
	\label{fig:entire}
\end{figure}

Now we consider the case that $\Gamma(D)$ is disconnected.
Then the original link diagram can be decomposed into several parts in the plane that can be separated by a simple closed curve disjoint from the diagram, which indicates that the diagram is a split diagram.
However, for each connected component one can construct a binding circle, and by taking their connected sum, one obtains the desired binding circle.
\end{proof}

Now we prove Theorem~\ref{thm:main}.

\begin{proof}[Proof of Theorem~\ref{thm:main}]
Let $L$ be an alternating link. 
Suppose $L$ admits an alternating diagram whose dual graph is bipartite. 
By Lemma~\ref{lem:minimal}, $L$ also has a minimal crossing diagram with a bipartite dual graph.  
Applying Theorem~\ref{thm:main2}, we then obtain a rotated circular three-page presentation of $L$ consisting of $3c(L)$ arcs. 
Finally, by Lemma~\ref{lem:length}, this yields the desired upper bound, and the proof is complete.
\end{proof}

Theorem~\ref{thm:main} not only strengthens the result of Kim, No, and Yoo, but also improves upon several known bounds for broad classes of alternating links.
If two alternating links $L_1$ and $L_2$ satisfy the conditions of the theorem and admit alternating diagrams with compatible crossing signs, then their connected sum $L_1 \# L_2$ admits an alternating diagram whose dual graph is bipartite.
Since the crossing number is additive for alternating links, the connected sum $L_1 \# L_2$ also satisfies the theorem.
Although connected sum is not our main focus, our method extends naturally to this setting.

\begin{corollary}
Let $L_1$ and $L_2$ be alternating links such that each admits an alternating diagram whose dual graph is bipartite.
Then
$$
\mathrm{Rib}(L_1 \# L_2) \le \sqrt{3}\, c(L_1 \# L_2).
$$
\end{corollary}

To the best of our knowledge, no previous work has addressed the behavior of folded ribbonlength under connected sum operations.
In particular, there have been no explicit upper bounds for the ribbonlength of a connected sum of links expressed in terms of the crossing number.

The significance of our result lies in deriving upper bounds on the ribbonlength from the rotated three-page presentation.
However, as the crossing number increases, links satisfying our assumptions become increasingly sparse.
To date, existing results related to ribbonlength have been systematically summarized by Chen et al.~\cite{CZDE}.
As immediate consequences of our main result, we obtain the following improved upper bounds on the ribbonlength.

\begin{corollary}
The following sharp upper bounds on the ribbonlength hold for the specified knots:
$$
\mathrm{Rib}(K)\leq
\begin{cases}
3\sqrt{3}, & \text{if } K \text{ is the trefoil knot},\\[4pt]
(n+2)\sqrt{3}, & \text{if } K = T_n \text{ with } n=1,3,5,\\[4pt]
6\sqrt{3}, & \text{if } K = P(2,2,2).
\end{cases}
$$
Here, $T_n$ denotes the twist knot with $n$ half twists, and $P(2,2,2)$ denotes the $(2,2,2)$-pretzel knot.
\end{corollary}

\section*{Declaration of generative AI and AI-assisted technologies in the manuscript preparation process}

During the preparation of this manuscript, the author used an AI-based language tool to improve grammar and readability. The author reviewed and revised all content and takes full responsibility for the final version of the manuscript.

\section*{Acknowledgements}
The author would like to thank Professor Hyoungjun Kim (Kyungpook National University) and Professor Sungjong No (Kyonggi University) for their valuable comments on the proof and their suggestions that contributed to the examples.
This study was supported by Basic Science Research Program of the National Research Foundation of Korea (NRF) grant funded by the Korea government Ministry of Education (RS-2023-00244488).

\bibliography{ribgen} 
\bibliographystyle{abbrv}

\end{document}